\documentclass[a4paper,12pt]{amsart}
 
\usepackage{amssymb, amsmath, amsthm, amsfonts}
\usepackage{extsizes}
\usepackage{hyperref}
\usepackage{xcolor,graphicx}

\usepackage{booktabs}
\theoremstyle{plain}
\newtheorem{theorem}{Theorem}[section]
\newtheorem{lemma}[theorem]{Lemma}
\newtheorem{corollary}[theorem]{Corollary}
\newtheorem{proposition}[theorem]{Proposition}

\theoremstyle{definition}
\newtheorem{definition}[theorem]{Definition}
\newtheorem{remark}[theorem]{Remark}

\def\r{\mathbb R}
 
 \def\sol{\text{Sol}_3}
\usepackage{xcolor,graphicx}

\begin{document}
\title[Gauss curvature flow solitons in Sol]{Gauss curvature solitons on invariant surfaces in the homogeneous space Sol}
\author{Rafael Belli}
\address{Department of Mathematics. Federal University of São Carlos. 13565-905 São Carlos, Brazil}
\email{rafaelbelli@estudante.ufscar.br}
\author{Rafael L\'opez}
\address{Department of Geometry and Topology. University of Granada. 18071 Granada, Spain}
\email{rcamino@ugr.es}
\maketitle

\begin{abstract}

We classify invariant surfaces in the 3-dimensional solvable Lie group $\sol$ that act as solitons for the Gauss curvature flow. We consider solitons associated with the canonical basis of Killing vector fields $\{F_1, F_2, F_3\}$, where $F_1$ and $F_2$ generate horizontal translations and $F_3$ generates the scaling isometry. We establish rigidity results for $F_3$-invariant surfaces, proving that specific totally geodesic vertical planes are the only $F_1$- and $F_2$-solitons. For $F_1$-invariant surfaces, we establish the main geometric properties of $F_2$- and $F_3$-solitons in both the extrinsic and intrinsic Gauss curvature.
\end{abstract}

\section{Introduction and statement of the main results}
 
The study of geometric flows has become a fundamental tool in modern Riemannian geometry, 
providing deep insights into the classification and structure of manifolds. Among these, the Gauss curvature flow (GCF), where a surface evolves with a normal velocity proportional to its Gauss curvature $K$, has been 
extensively studied since its introduction by Firey to model the shape of tumbling stones \cite{firey}. A smooth immersion $\psi: \Sigma \to \sol$ evolves by the Gauss curvature flow if it satisfies $\frac{\partial \Psi}{\partial t} = -K N$, where $K$ is the (extrinsic or intrinsic) Gauss curvature and $N$ is the unit normal vector field. Seminal papers on the GCF in Euclidean space include \cite{an,ch,ts,ur}.

A particularly interesting class of solutions are the solitons, which are surfaces whose shape remains fixed along the flow and evolve purely by the isometries of the ambient space. If we consider the isometries generated by a Killing vector field $X$, then solitons are characterized by the PDE 
\begin{equation}\label{eqsol}
K = -\langle N,X\rangle.
\end{equation}
In this paper, we study solitons of the GCF in the space $\sol$. The space $\sol$ is one of the eight model geometries of Thurston, defined as a simply connected 
3-dimensional solvable Lie group equipped with a left-invariant metric. Unlike spaces of constant sectional curvature, $\sol$ possesses a more rigid isometry group of dimension 3, and its sectional curvatures are not constant.

The study of solitons is difficult in all its generality. In order to simplify this problem, it is natural to assume a certain type of invariance of the surface that allows the PDE \eqref{eqsol} to be reduced to an ODE. In the case of $\sol$, an interesting family of surfaces consists of those that are invariant under a one-parameter group of isometries of $\sol$. Since the isometry group is of dimension $3$, essentially there are three types of invariant surfaces. Invariant surfaces with constant mean curvature or constant Gauss curvature have been studied in the literature: see \cite{lo1,lo2,lm1,lm2,lm3,vr}.

To the best of the authors' knowledge, the study of the GCF in $\sol$ has not been considered yet. Regarding another geometric flow of great interest, the mean curvature flow, it is worth pointing out the remarkable paper by Pipoli, where he classified all invariant surfaces that are solitons under the mean curvature flow \cite{pi}. The present paper is an extension of this line of research to the GCF. Moreover, we will consider solitons where $K$ is either the extrinsic or the intrinsic Gauss curvature of the surface. 

In order to precisely state the results proved in this paper, we need to introduce the $\sol$ space. Consider the model of $\sol$ as the Euclidean space $\r^3$ endowed with a group operation 
 defined by
 $$(x, y, z) \ast (x', y', z') = (x + e^{-z}x', y + e^z y', z + z'),$$
 where $(x,y,z)$ are Cartesian coordinates in $\r^3$. The metric that makes $\sol$ a Riemannian manifold is the left-invariant metric given by
 $$ds^2 = e^{2z}dx^2 + e^{-2z}dy^2 + dz^2.$$
In the soliton equation \eqref{eqsol}, we can replace $X$ with elements from a basis of Killing vector fields of $\sol$, which is formed by 
$$F_1(x,y,z)=\partial_x,\quad F_2=\partial_y,\quad F_3(x,y,z)=-x\partial_x+y\partial_y+\partial_z.$$

\begin{definition} A surface $\Sigma$ in $\sol$ is called an $F_k$-soliton, $1\leq k\leq 3$, if $\Sigma$ satisfies
\begin{equation}\label{eqsol2}
K = -\langle N,F_k\rangle,
\end{equation}
where $N$ is the unit normal vector field of $\Sigma$ and $K$ is the extrinsic ($K_{ext}$) or intrinsic ($K_{int}$) Gauss curvature of $\Sigma$.
\end{definition} 

On the other hand, in $\sol$ there are three types of surfaces invariant under a one-parameter group of isometries, corresponding to the isometries generated by the Killing vector fields $F_k$. According to each case, the invariant surfaces can be parametrized by 
\begin{align*}
(s,t)\mapsto& (t, y(s), z(s)), & \mbox{$F_1$-invariant}\\
(s,t)\mapsto & (x(s),t ,z(s)), & \mbox{$F_2$-invariant}\\
(s,t)\mapsto& (x(s)e^{-t}, y(s)e^t, t), & \mbox{$F_3$-invariant}.
\end{align*}
The isometry of $\sol$ defined by $(x,y,z)\mapsto (y,x,-z)$ maps $F_1$-invariant surfaces into $F_2$-invariant surfaces and {\it vice versa}. In particular, we can restrict our study of invariant surface to only $F_1$ and $F_3$-invariant surfaces. 

A consequence of the definition of an invariant surface is that the vector field $F_k$ is tangent to an $F_k$-invariant surface. Thus, the right-hand side of \eqref{eqsol2} vanishes identically. This proves that an $F_k$-invariant surface is an $F_k$-soliton if and only if it is extrinsically or intrinsically flat (Theorem \ref{t34}). 

Although the study of solitons under the GCF in $\sol$ is restricted to $F_1$- and $F_3$-invariant surfaces, each one has two types of solitons, and for each case, there are two types of Gauss curvature, namely, extrinsic and intrinsic. This yields a total of eight cases to study. However, for $F_3$-invariant surfaces, we have obtained a full classification in Section \ref{s4}. We have proved that the only extrinsic $F_1$-soliton (resp. $F_2$-soliton) is the vertical plane $y = 0$ (resp. $x=0$). Also, we prove that there are no intrinsic $F_1$-solitons nor $F_2$-solitons (Theorems \ref{t41} and \ref{t42}).
 
For $F_1$-invariant solitons, we have given a geometric description of the surfaces in Sections \ref{s5} and \ref{s6}, indicating their main properties. We have obtained an explicit parametrization of $F_2$-solitons in terms of an implicit equation (Theorems \ref{t54} and \ref{t57}). In the case of $F_3$-solitons, we demonstrate that horizontal and vertical planes 
are the only trivial solutions, while any other generating curve inevitably reaches 
a vertical tangent in finite arc-length, preventing its smooth extension: see Theorems \ref{t64} and \ref{t65}.

Before these results, we have deduced the parametrizations of invariant surfaces by calculating their Gauss curvature (Section \ref{s2}). In Section \ref{s3}, we derive the soliton equation \eqref{eqsol2}.
 \section{Invariant surfaces in $\sol$ and its Gauss curvature}\label{s2}
 
For the sake of completeness of the paper, in this section we recall the basic Riemannian geometry of $\sol$, emphasizing the space of Killing vector fields, and we deduce the invariant surfaces of $\sol$. Finally, we compute the Gauss curvature of these surfaces. These results are well-known in the literature. For example, invariant surfaces appear in \cite{lo2,lm1,lm3}. The calculation of the Gauss curvature for $F_1$-invariant surfaces is done in \cite{lm3}.

Let $\{e_{1}=\partial_x, e_{2}=\partial_y,e_{3}=\partial_z \} $ be the canonical basis of vector fields of $\r^{3}$. 
 A global orthonormal frame for the Lie algebra of $\sol$ consists of the left-invariant vector fields
 $$E_1 = e^{-z}e_1, \quad E_2 = e^ze_2, \quad E_3 =e_3.$$
 Using Koszul’s formula, the Levi-Civita connection $\nabla$ of the metric $g$ with respect to the orthonormal frame $\{E_1, E_2, E_3\}$ is given by
\begin{equation}\label{nabla}
\begin{array}{ccc}
\nabla_{E_1}E_1 = -E_3, & \nabla_{E_1}E_2 = 0, & \nabla_{E_1}E_3 = E_1 \\
\nabla_{E_2}E_1 = 0, & \nabla_{E_2}E_2 = E_3, & \nabla_{E_2}E_3 = -E_2 \\
\nabla_{E_3}E_1 = 0, & \nabla_{E_3}E_2 = 0, & \nabla_{E_3}E_3 = 0.
\end{array}
\end{equation}
 The isometry group of $\sol$ has dimension 3. The identity component is generated by the left translations
\[
L_{( j,k,l) }:( x,y,z) \in \sol\rightarrow
( j,k,l) \ast ( x,y,z) \in \sol\text{,}
\]%
 hence $\mbox{Iso}(\sol)$ is 3-dimensional. Given $p=( x,y,z) \in \sol$, to find a basis for this isometry group, let us consider a smooth curve $%
\alpha _{v}:( -\delta ,\delta ) \rightarrow \sol$
given by $\alpha _{v}( t) =tv$. Since the map 
\begin{equation*}
\Phi
_{p}:v\in \r^{3}\rightarrow \frac{d}{dt}{\Big|}_{t=0}( L_{\alpha
_{v}( t) }p) \in \mathfrak{X}( \sol) 
\end{equation*}
is an isomorphism onto its image (the space of Killing vector fields), a
 basis of Killing vector field is $\left\{ F_{k}(
p) :=\Phi _{p}( e_{k})\colon 1\leq k\leq 3 \right\}$. Note
that%
\begin{eqnarray*}
\Phi_{\alpha _{e_{1}}( t) }p &=&( t,0,0) \ast (
x,y,z) =( t+x,y,z) \\
\Phi_{\alpha _{e_{2}}( t) }p &=&( 0,t,0) \ast (
x,y,z) =( x,t+y,z) \\
\Phi_{\alpha _{e_{2}}( t) }p &=&( 0,0,t) \ast (
x,y,z) =( e^{-t}x,e^{t}y,t+z) .
\end{eqnarray*}%
Differentiating and evaluating at $t=0$, a basis for all
Killing vector fields in $\sol$ i given by 
 \begin{equation}\label{kf}
 \begin{split}
 F_1(x,y,z)&=e_1=e^z E_1,\\
 F_2(x,y,z)&=e_2=e^{-z}E_2\\
 F_3(x,y,z)&=-xe_1+ye_2+e_3=-xe^z E_1+ye^{-z}E_2+E_3.
 \end{split}
 \end{equation}

 \begin{definition} A surface $\Sigma$ in $\sol$ is said to be invariant under the action of a one-parameter group of isometries $\{\Phi_t\colon t\in\r\}$ if $\Phi_t(\Sigma) = \Sigma$ for all $t$. If this group is determined by the Killing vector field $F_k\in\mathfrak{X}(\sol)$, we say that $\Sigma$ is $F_k$-invariant. 
 \end{definition}

In order to simplify the notation and present the results clearly, we will restrict our study to invariant surfaces for the basis $\{F_1,F_2,F_3\}$ of Killing vector fields. Furthermore, the isometry of $\sol$ given by 
$$(x,y,z)\mapsto (y,x,-z)$$
maps $F_1$-invariant surfaces to $F_2$-invariant surfaces and {\it vice versa}. Therefore, it is enough to consider $F_k$-invariant surfaces for $k=1,3$. 

Parametrizations of $F_k$-invariant surfaces are the following (\cite{lm1,lm3}): 
\begin{align}
\Psi(s, t) &= (t, y(s), z(s)), &k=1,\label{para1}\\
\Psi(s, t)& = (x(s)e^{-t}, y(s)e^t, t), & k=3. \label{para3}
\end{align}
where $x,y,z\colon I\subset \r$ are smooth functions and $t\in\r$. 
 
 We compute the Gauss curvature of both types of invariant surfaces. Although the calculations are known (\cite{lm3}), we also provide the explicit expression for the unit normal vector field $N$, which is necessary in the soliton equation \eqref{eqsol2}.

 \begin{proposition} 
Let $\Sigma$ be an $F_1$-invariant surface in $\sol$, parametrized by \eqref{para1}. Then $K_{ext}$ and $K_{int}$ of $\Sigma$ are given by
\begin{align*}
K_{ext} &= -\frac{y' e^{2z}( y' z'' - y'' z' + 2y'z'^2 + y'^3 e^{-2z})}{W^4}, \\
K_{int} &= K_{ext} + \frac{y'^2 - z'^2 e^{2z}}{W^2},
\end{align*}
where $W =\sqrt{ y'^2 + z'^2 e^{2z}}$, and the prime denotes differentiation with respect to the parameter $s$.
\end{proposition}

\begin{proof}
We have $\Psi_t = \partial_x $ and $\Psi_s = y' \partial_y + z' \partial_z$. The coefficients of the first fundamental form are 
$$
E = \langle \Psi_t, \Psi_t \rangle = e^{2z} \quad F = \langle \Psi_t, \Psi_s \rangle = 0 \quad G = \langle \Psi_s, \Psi_s \rangle = y'^2 e^{-2z} + z'^2.$$
Let $W=\sqrt{EG-F^2}=\sqrt{ y'^2 + z'^2 e^{2z} }$.
The unit normal vector field $N$ is 
\begin{equation}\label{n1}
N = \frac{1}{W} \left( z' e^z E_2 - y' E_3 \right).
\end{equation}

Next, we compute the coefficients of the second fundamental form $e, f, g$. We use the Levi-Civita connection $\nabla$ of $\text{Sol}_3$. Using \eqref{nabla}, we have 
\begin{equation*}
\nabla_{\partial_x} \partial_x = -e^{2z} \partial_z, \quad \nabla_{\partial_y} \partial_y = e^{-2z} \partial_z, \quad \nabla_{\partial_x} \partial_z = \partial_x, \quad \nabla_{\partial_y} \partial_z = -\partial_y
\end{equation*}
Taking the second derivatives of $\Psi$, we find
\begin{align*}
\nabla_{\Psi_t} \Psi_t &= -e^{2z} \partial_z \\
\nabla_{\Psi_t} \Psi_s &= z' \partial_x \\
\nabla_{\Psi_s} \Psi_s &= (y'' - 2y'z') \partial_y + \left( z'' + y'^2 e^{-2z} \right) \partial_z.
\end{align*}
Therefore 
\begin{align*}
e &= \langle \nabla_{\Psi_t} \Psi_t, N \rangle = \frac{y' e^{2z}}{W}, \\
f &= \langle \nabla_{\Psi_t} \Psi_s, N \rangle = 0, \\
g &= \langle \nabla_{\Psi_t} \Psi_s, N \rangle = - \frac{1}{W} \left( y' z'' - y'' z' + 2y'z'^2 + y'^3 e^{-2z} \right).
\end{align*}

The extrinsic Gauss curvature is given by $K_{ext} = \frac{eg - f^2}{EG - F^2}$. Since $f=0$, we have:
\begin{equation*}
K_{ext} = -\frac{y' e^{2z} \left( y' z'' - y'' z' + 2y'z'^2 + y'^3 e^{-2z} \right)}{W^4}.
\end{equation*}
To determine the intrinsic Gauss curvature $K_{int}$, we use the Gauss equation
\begin{equation*}
K_{int} = K_{ext} + \bar{K}(\Pi),
\end{equation*}
where $\bar{K}(\Pi)$ is the ambient sectional curvature of $\text{Sol}_3$ evaluated on the tangent plane $\Pi$ of $\Sigma$. 
 For a plane $\Pi$ in $\text{Sol}_3$ which is determined by its unit normal vector, $N = n_1 E_1 + n_2 E_2 + n_3 E_3$, the sectional curvature $\bar{K}(\Pi)$ is computed by the formula $\bar{K}(\Pi) = K_{23} n_1^2 + K_{13} n_2^2 + K_{12} n_3^2$, 
where $K_{ij}$ denotes the sectional curvature of the plane spanned by $\{E_i, E_j\}$. Since, $K_{12} = 1$, $K_{13} = -1$ and $K_{23} = -1$, and from \eqref{n1}, we arrive at $
\bar{K}(\Pi) = \frac{y'^2 - z'^2 e^{2z}}{W^2}$. This completes the proof.
\end{proof}

\begin{corollary} \label{cor23}
Let $\Sigma$ be an $F_1$-invariant surface in $\sol$, parametrized by \eqref{para1}. Assume the generating curve $\gamma(s) = (0, y(s), z(s))$ is parameterized by arc length, which implies that there exists a smooth function $\theta(s)$ such that
\begin{equation}\label{1pp}
\begin{split}
y'(s) e^{-z(s)} &= \cos\theta(s), \\
z'(s) &= \sin\theta(s). 
\end{split}
\end{equation}
Under these conditions, the extrinsic and intrinsic Gauss curvatures of $\Sigma$ are
\begin{align*}
K_{ext} &= -\theta' \cos\theta -\cos^2\theta, \\
K_{int} &= -\theta'\cos\theta - \sin^2\theta.
\end{align*}
\end{corollary}

\begin{proposition} \label{pr26}
Let $\Sigma$ be an $F_3$-invariant surface in $\sol$, parametrized by \eqref{para3}. Then, the extrinsic Gauss curvature $K_{ext}$ and the intrinsic Gauss curvature $K_{int}$ of $\Sigma$ are respectively given by
\begin{align*}
K_{ext} &= \frac{eg - f^2}{W^2}, \\
K_{int} &= K_{ext} + \frac{(xy' + yx')^2 - x'^2 - y'^2}{W^2},
\end{align*}
where $W =\sqrt{ x'^2 + y'^2 + (xy' + yx')^2}$, and the coefficients of the second fundamental form are:
\begin{align*}
e &= -\frac{1}{W} \left[ xy' - yx' - (y^2 - x^2)(xy' + yx') \right], \\
f &= \frac{1}{W} \left[ (xx' + yy')(xy' + yx') \right], \\
g &= -\frac{1}{W} \left[ x'y'' - x''y' - (y'^2 - x'^2)(xy' + yx') \right].
\end{align*}
\end{proposition}

\begin{proof}
The tangent plane is spanned by 
$$\Psi_t = -x E_1 + y E_2 + E_3, \quad \Psi_s= x' E_1 + y' E_2.$$
The coefficients of the first fundamental form are
$$E = x^2 + y^2 + 1, \quad F = -xx' + yy', \quad 
G = \langle \Psi_s, \Psi_s \rangle = x'^2 + y'^2.$$
Now $W=\sqrt{ x'^2 + y'^2 + (xy' + yx')^2}$. The unit normal vector field $N$ is 
\begin{equation}\label{n3}
N = \frac{1}{W} \left( y' E_1 - x' E_2 +(xy' + yx') E_3 \right).
\end{equation}
On the other hand, we have
\begin{align*}
\nabla_{\Psi_t} \Psi_t &= -x E_1 - y E_2 + (y^2 - x^2) E_3, \\
\nabla_{\Psi_t} \Psi_s &= (xx' + yy') E_3, \\
\nabla_{\Psi_s} \Psi_s &= x'' E_1 + y'' E_2 + (y'^2 - x'^2) E_3,
\end{align*}
which, in combination with \eqref{n3} gives the desired expressions for $e, f, g$. For the intrinsic Gauss curvature, we compute $\bar{K}(\Pi)$. From \eqref{n3} and the expressions $K_{ij}$, we have $\bar{K}(\Pi) = \frac{1}{W^2} \left( -y'^2 - x'^2 + (xy' + yx')^2 \right)$.

\end{proof}

\begin{corollary} \label{cor25} 
Let $\Sigma$ be an $F_3$-invariant surface parametrized by \eqref{para3}. Assume the generating curve $\gamma(s) = (x(s), y(s), 0)$ is parameterized by arc length, which implies that there exists a smooth function $\theta(s)$ such that 
\begin{equation}\label{3pp}
\begin{split}
x'(s) &= \cos\theta(s), \\
y'(s)& = \sin\theta(s).
\end{split}
\end{equation} 
 Define the auxiliary function $u(s) = x(s)\sin\theta(s) + y(s)\cos\theta(s)$. Under these conditions, the coefficients of the second fundamental form are
\begin{align*}
e &= \frac{-x\sin\theta +y\cos\theta + u(y^2 - x^2)}{\sqrt{1+u^2}}, \\
f &= \frac{u(x\cos\theta + y\sin\theta)}{\sqrt{1+u^2}}, \\
g &= -\frac{\theta' + u(\cos^2\theta - \sin^2\theta)}{\sqrt{1+u^2}}.
\end{align*}
Using this, we have
$$K_{ext} = \frac{eg - f^2}{1 + u^2}, \quad K_{int} = K_{ext} + \frac{u^2 - 1}{1 + u^2}.$$
\end{corollary}

\begin{remark} The angle $\theta$ in the two curves parametrized by arc-length that appear in \eqref{1pp} and \eqref{3pp} has the following geometric meaning. For the curve \eqref{1pp}, its curvature is $\kappa= |\theta'(s) + \cos\theta(s)|$. For the curve in \eqref{3pp}, its geodesic curvature as a curve in the plane $z=0$ is $\kappa_g=\theta'(s)$ and its curvature as a curve in $\sol$ is $\kappa = \sqrt{(\theta')^2 + \cos^2(2\theta)}$.
\end{remark}

\section{Solitons of Gauss curvature flow}\label{s3}

Given an initial smooth surface $\psi\colon\Sigma\to \sol$, the Gauss curvature flow is a smooth map $\Psi
:( -\varepsilon ,\varepsilon ) \times \Sigma\rightarrow \sol$ that satisfies the following PDE
\begin{equation}
\left\{ 
\begin{split}
\frac{\partial \Psi }{\partial t}( t,p)& =-K( t,p),
N( t,p) \\ 
\Psi ( 0,p)& =\psi ( p).
\end{split}%
\right. \label{eq1}
\end{equation}
Let $X\in\mathfrak{X}(\sol)$ be a Killing vector field and let $\{\Phi_t\colon t\in\r\}$ be the one-parameter group of isometries associated with $X$, that is, $\Phi_t\colon\sol\to\sol$ is an isometry such that $\Phi_0=\mbox{id}$ and 
\begin{equation}\label{iso}
\frac{\partial \Phi_t(p) }{\partial t} =X( \Phi_t(p)), \quad t\in\r, p\in\Sigma.
\end{equation} Let us write
$$\Phi_t(p)=\Phi(t,p).$$

 \begin{definition} 
A family of immersions $\Psi: \Sigma \times [0, T) \to \sol$ evolving by the GCF is said to be a soliton if there exists a one-parameter group of isometries $\{\Phi_t\colon t\in\r\}$ of $\sol$ generated by a Killing vector field $X$, such that the image $\Sigma_t = \Psi(\Sigma, t)$ of $\Sigma$ under the flow satisfies
\begin{equation*}
\Sigma_t = \Phi_t(\Sigma_0)
\end{equation*}
for all $t \in [0, T)$, where $\Sigma_0 = \Psi(\Sigma, 0)$ is the initial surface. 
\end{definition}

To derive the stationary equation for such a soliton, we note that the condition $\Sigma_t = \Phi_t(\Sigma_0)$ implies that, up to a time-dependent reparametrization (diffeomorphism) $\sigma_t: \Sigma \to \Sigma$ with $\sigma_0 = id_\Sigma$, the flow can be written as
\begin{equation*}
\Psi(t, p) = \Phi(t, \sigma_t(p)).
\end{equation*}
Differentiating this expression with respect to $t$ and evaluating at $t=0$, we obtain by the chain rule
\begin{equation*}
\begin{split}
\frac{\partial \Psi}{\partial t}(0, p) &= \frac{\partial \Phi}{\partial t}(0, \sigma_0(p)) + (d\Phi_0)_p \left( d\psi \left( \left. \frac{\partial \sigma_t(p)}{\partial t} \right|_{t=0} \right) \right)\\
&=X(p) + d\psi ( V(p) ),
\end{split}
\end{equation*}
where we have used \eqref{iso}, $d\Phi_0 = \mbox{id}$, and we have set $V(p)= \left. \frac{\partial \sigma_t(p)}{\partial t} \right|_{t=0} $.
 Substituting into \eqref{eq1}, we have
$$
X (p)+ d\psi_p(V_p)= -K(p)N(p),
$$
for all $p\in\Sigma$. 
Multiplying by $N$ and taking into account that $d\psi_p(V_p)$ is tangent to $\Sigma$, we arrive at the soliton equation \eqref{eqsol2}. We have proved the following resut: 

\begin{proposition} Let $\Sigma$ be a surface in $\sol$ and let $X$ be a Killing vector field. Then $\Sigma$ is a soliton with respect to $X$ if and only if $\Sigma$ satisfies
\begin{equation}\label{eq2}
K=-\langle N,X\rangle .
\end{equation}
\end{proposition}

In the following result, and for invariant surfaces, we obtain the expressions for $\langle N,X\rangle$ for invariant surfaces, that is, the right-hand sides of \eqref{eq2}. The proof is immediate from \eqref{kf}, \eqref{n1} and \eqref{n3}.

\begin{lemma} Let $\Sigma$ be an $F_k$-invariant surface, $k=1,3$, parametrized by \eqref{para1} or \eqref{para3}, and such that the generating curves are parametrized by \eqref{1pp} or \eqref{3pp}.
\begin{enumerate}
\item If $k=1$, then 
 \begin{equation}\label{eq17}
 \begin{split}
 \langle N,F_1\rangle&=0,\\
 \langle N,F_2\rangle&=\sin\theta\, e^{-z},\\
 \langle N,F_3\rangle&=y\sin\theta\, e^{-z}-\cos\theta.
 \end{split}
 \end{equation}
\item If $k=3$, then
 \begin{equation}\label{eq18}
 \begin{split}
 \langle N,F_1\rangle&=\frac{\sin\theta}{\sqrt{1+u^2}} e^t,\\
 \langle N,F_2\rangle&=-\frac{\cos\theta}{\sqrt{1+u^2}} e^{-t},\\
 \langle N,F_3\rangle&=0.
 \end{split}
 \end{equation}
\end{enumerate}
\end{lemma} 
 
As a consequence of the fact that $\langle N,F_k \rangle=0$ for an $F_k$-invariant surface, the following result is immediate. 
 
 \begin{theorem}\label{t34}
An $F_k$-invariant surface in $\sol$, $1\leq k\leq 3$, is an $F_k$-soliton of the GCF if and only if it is flat. 
\end{theorem}
 
We finish this section by providing in Table \ref{ex1}, examples of invariant solitons of the GCF. To verify this, we use Corollaries \ref{cor23} and \ref{cor25}, together with \eqref{eq17} and \eqref{eq18}.

 \begin{table}[h]
\centering

\begin{tabular}{@{}llccc@{}}
\toprule
 & & \multicolumn{3}{c}{ Soliton type } \\ \cmidrule(l){3-5} 
Invariance & Gauss curvature & $F_1$ & $F_2$ & $F_3$ \\ \midrule
 {$F_1$-Invariant} & Intrinsic & --- & $z=c$ & --- \\
 & Extrinsic & --- & --- & $z=c, y=0$ \\ \addlinespace
 {$F_3$-Invariant} & Intrinsic & --- & --- & --- \\
 & Extrinsic & $y=0$ & $x=0$ & --- \\ \bottomrule
\end{tabular}
\caption{Solitons obtained by planes parallel to the coordinate planes. }\label{ex1}
\end{table}
\section{$F_3$-invariant surfaces: the case of $F_1$ and $F_2$-solitons}\label{s4}

In this section, we obtain a full classification of all $F_3$-invariant solitons of the GCF. The key to these arguments is that the left-hand side of \eqref{eq2} is independent of the parameter $t$ due to the invariance of the surface. In contrast, the right-hand side, the inner product $\langle N,F_k\rangle$, depends on $t$. In the first result, we consider $F_1$-solitons, and in the second one, $F_2$-solitons.

\begin{theorem}\label{t41}
The vertical plane $y=0$ is the only $F_3$-invariant extrinsic $F_1$-soliton. Furthermore, there are no $F_3$-invariant intrinsic $F_1$-solitons.
\end{theorem}

\begin{proof} Let $\Sigma$ be an $F_3$-invariant surface parametrized by \eqref{para3} such that the generating curve $\gamma$ is parametrized by \eqref{3pp}.

Suppose $\Sigma$ is an $F_1$-soliton (either extrinsic or intrinsic). In \eqref{eq18}, we have the expression for $\langle N,F_1\rangle$. Thus, \eqref{eq2} is 
\begin{equation}\label{con3}
K = -\frac{\sin\theta(s)}{\sqrt{1+u(s)^2}}e^t,
\end{equation}
where $K \in \{K_{ext}, K_{int}\}$ is the corresponding Gauss curvature of the surface. However, $K_{ext}$ and $K_{int}$ depend only on the arc length parameter $s$ (Corollary \ref{cor25}). Thus, the right-hand side in \eqref{con3} must be zero. This implies $$\sin\theta(s) = 0 \quad \text{for all } s,$$
and consequently, $K=0$.

We now investigate the two types of Gauss curvature. From $\sin\theta = 0$ and $\cos\theta = \pm 1$, we deduce that $x'(s) = \pm 1$ and $y'(s) = 0$. Integrating these relations yields $x(s) = \pm s + x_0$ and $y(s) = y_0$ for some constants $x_0, y_0 \in \r$. Then $\Sigma$ is the vertical plane $y=y_0$. 

The auxiliary function $u(s) = x\sin\theta + y\cos\theta$ becomes $u(s) = \pm y_0$, which is constant. The denominator is $W = \sqrt{1+y_0^2}$. Using the formulas for the second fundamental form with $\theta' = 0$ and constant $u$, we obtain 
$$e = \pm \frac{y_0(1+y_0^2-x^2)}{W},\quad f = \frac{y_0 x}{W},\quad g = \mp \frac{y_0}{W}.$$
 The numerator of $K_{ext}$ is $eg - f^2 = -y_0^2$. Thus, the extrinsic and intrinsic curvatures are constants given by
$$K_{ext} = \frac{-y_0^2}{1+y_0^2}, \quad K_{int} = K_{ext} + \frac{y_0^2 - 1}{1+y_0^2} = \frac{-1}{1+y_0^2}.$$
Thus, $K_{ext}=0$ implies $y_0=0$ and $\gamma(s) = (\pm s + x_0, 0, 0)$. Then $\Sigma$ is the vertical plane $y = 0$. The converse is known such is shown in Table \ref{ex1}.

The equation $K_{int} = 0$ yields a contradiction, concluding that no such intrinsic $F_1$-soliton exists.
\end{proof}

 \begin{theorem}\label{t42}
 The vertical plane $x=0$ is the only $F_3$-invariant extrinsic $F_2$-soliton. Furthermore, there are no $F_3$-invariant intrinsic $F_2$-solitons.
\end{theorem}

\begin{proof}
By \eqref{eq18}, the soliton equation \eqref{eq2} is
$$K = \frac{\cos\theta(s)}{\sqrt{1+u(s)^2}}e^{-t},$$
where $K \in \{K_{ext}, K_{int}\}$. Since both $K_{ext}$ and $K_{int}$ depend only on the arc length parameter $s$, and the right-hand side of the equation is a function of $s$ alone, we must have
$$\cos\theta(s) = 0 \quad \text{for all } s,$$
and $K=0$. 

From $\cos\theta = 0$ and $\sin\theta = \pm 1$, we deduce that $x'(s) = 0$ and $y'(s) = \pm 1$. Integrating these relations yields $x(s) = x_0$ and $y(s) = \pm s + y_0$ for some constants $x_0, y_0 \in \r$. Then $\Sigma$ is the plane $x=x_0$. The auxiliary function $u(s) = x\sin\theta + y\cos\theta$ becomes $u(s) = \pm x_0$, which is constant. The denominator is $W = \sqrt{1+x_0^2}$.

Using the formulas for the second fundamental form with $\theta' = 0$ and constant $u$, we have
$$e = \frac{\mp x_0(1 + x_0^2 - y(s)^2)}{W}, \quad f = \frac{x_0 y(s)}{W}, \quad g = \frac{\pm x_0}{W}.$$
This gives $eg - f^2 = -x_0^2$. Thus, 
$$K_{ext} = \frac{-x_0^2}{1+x_0^2}, \quad K_{int} = K_{ext} + \frac{x_0^2 - 1}{1+x_0^2} = \frac{-1}{1+x_0^2}.$$
 If $\Sigma$ is an extrinsic $F_2$-soliton, then $x_0=0$. The generating curve is $\gamma(s) = (0, \pm s + y_0, 0)$ and the corresponding invariant surface is the vertical plane $x = 0$. Moreover, this plane is an extrinsic $F_2$-soliton (see Table \ref{ex1}).

 If $\Sigma$ is an intrinsic $F_2$-soliton, since $K_{int} = 0$, we arrive at a contradiction, confirming that no such intrinsic $F_2$-soliton exists.
\end{proof}
 
 We emphasize the case of intrinsic solitons.
 
 \begin{corollary} There are no $F_3$-invariant intrinsic $F_1$- and $F_2$-solitons of the GCF in $\sol$.
 \end{corollary}

 \section{$F_1$-invariant surfaces: the case of $F_2$-solitons}\label{s5}

We begin with the study of $F_1$-invariant solitons of the GCF. In this section, we focus on $F_2$-solitons, and in the next one, we deal with $F_3$-solitons. For $F_2$-solitons, we express the soliton equation \eqref{eq2} in terms of the generating curve. The following result is straightforward from \eqref{eq17} and Corollary \ref{cor23}.

\begin{proposition} \label{pr61}
Let $\Sigma$ be an $F_1$-invariant surface in $\sol$ generated by a curve parametrized by arc-length according to \eqref{1pp}. 
Suppose that $\Sigma$ is an $F_2$-soliton.
\begin{enumerate}
 \item If $\Sigma$ is an extrinsic soliton, then $\theta$ and $z$ satisfy the equation
 \begin{equation}\label{121}
 \theta' \cos\theta+ \cos^2\theta = \sin\theta\, e^{-z}.
 \end{equation}
 \item If $\Sigma$ is an intrinsic soliton, then $\theta$ and $z$ satisfy the equation
 \begin{equation}\label{122}
 \theta'\cos\theta + \sin^2\theta = \sin\theta\, e^{-z}.
 \end{equation}
\end{enumerate}
\end{proposition}

From now on, we will assume that the generating curve of $F_1$-invariant surfaces is parametrized according to \eqref{1pp}.

In the following two results, we investigate the case where the angle function $\theta$ is linear. We distinguish whether or not $\theta$ is a constant function. 

\begin{theorem}\label{t52}
Let $\Sigma$ be an $F_1$-invariant surface in $\sol$ generated by a curve $\gamma$. Suppose that the angle function is a constant, $\theta(s)=c$.
\begin{enumerate}
 \item The surface $\Sigma$ cannot be an extrinsic $F_2$-soliton.
 \item If $\Sigma$ is an intrinsic $F_2$-soliton, then $\Sigma$ is a horizontal plane $z = z_0$.
\end{enumerate}
\end{theorem}

\begin{proof}
 \begin{enumerate}
\item Suppose $\Sigma$ is an extrinsic $F_2$-soliton. 
Substituting $\theta = c$ and $\theta' = 0$ in \eqref{121}, we get $\cos^2 c = \sin c\, e^{-z(s)}$. In consequence, $\sin c \neq 0$. Since the right-hand side is constant, $z(s)$ must be constant, yielding $z'(s) = 0$. However, this is a contradiction because $z'(s) = \sin c \neq 0$. Thus, no such surface exists.

\item Suppose $\Sigma$ is an intrinsic $F_2$-soliton. Substituting $\theta = c$ in \eqref{122}, we obtain $\sin^2 c = \sin c\, e^{-z(s)}$. 
If $\sin c \neq 0$, then $z(s)$ is constant, but this gives a contradiction with the identity $z'(s) = \sin c$.
Therefore, we must have $\sin c = 0$. In this case, $z'(s) = 0$, meaning $z(s) = z_0$ for some constant $z_0$. The condition $y'(s) = e^{z(s)}\cos c$ becomes $y'(s) = \pm e^{z_0}$, which integrates to $y(s) = \pm s e^{z_0} + y_0$. The curve $\gamma$ is a straight line that generates a horizontal plane $z = z_0$. 
\end{enumerate}
\end{proof}

\begin{theorem}
Let $\Sigma$ be an $F_1$-invariant surface in $\sol$ generated by a curve $\gamma$. Suppose the angle function satisfies $\theta'(s) = c$ for some non-zero constant $c \in \r$. Then $\Sigma$ can neither be an extrinsic nor an intrinsic $F_2$-soliton.
\end{theorem}

\begin{proof}
 Assume $\theta'(s) = c \neq 0$, which implies that $\theta(s) = cs + \theta_0$. 

\begin{enumerate}
\item Suppose $\Sigma$ is an extrinsic $F_2$-soliton. 
Equation \eqref{121} becomes $c\cos\theta + \cos^2\theta = \sin\theta\, e^{-z}$. 
Since $\theta(s)$ is non-constant, we deduce that $\sin\theta \neq 0$ almost everywhere. Differentiating both sides gives
$$c(-c\sin\theta - 2\sin\theta\cos\theta) = c\cos\theta\, e^{-z} - \sin^2\theta\, e^{-z}.$$
To eliminate the $e^{-z}$ terms, we multiply \eqref{121} by $\sin\theta$ and substitute $\sin\theta\, e^{-z} = c\cos\theta + \cos^2\theta$ from the original equation,
$$-c^2\sin^2\theta - 2c\sin^2\theta\cos\theta = c\cos\theta(c\cos\theta + \cos^2\theta) - \sin^2\theta(c\cos\theta + \cos^2\theta).$$
This equation reduces to $-8 c^2-8 c \cos \theta -\cos (4 \theta )+1=0$. This is a contradiction. 

\item Suppose $\Sigma$ is an intrinsic $F_2$-soliton. Equation \eqref{122} gives $c\cos\theta + \sin^2\theta = \sin\theta\, e^{-z}$. 
Differentiating both sides yields
$$c(-c\sin\theta + 2\sin\theta\cos\theta) = c\cos\theta\, e^{-z} - \sin^2\theta\, e^{-z}.$$
Multiplying by $\sin\theta$ and substituting $\sin\theta\, e^{-z} = c\cos\theta + \sin^2\theta$ gives
$$-c^2\sin^2\theta + 2c\sin^2\theta\cos\theta = c\cos\theta(c\cos\theta + \sin^2\theta) - \sin^2\theta(c\cos\theta + \sin^2\theta),$$
or equivalently, 
$$ -8 c^2+4 c \cos \theta-4 c \cos (3 \theta )-4 \cos (2 \theta )+\cos (4 \theta )+3 =0.$$
This is a contradiction. 
\end{enumerate}
\end{proof}

 We now study the general case for the angle function $\theta(s)$.
 
 \begin{theorem}\label{t54}
Let $\gamma(s)=(y(s),z(s))$ be a generating curve of an $F_1$-invariant extrinsic $F_2$-soliton. Then $\gamma$ is a global graph on the $y$-axis, and the function $z$ is given by the implicit integral
\begin{equation}\label{eqf1}
\int \frac{e^{-z}}{1 + W\left( A e^{-\frac{1}{2} e^{-2z}} \right)} dz = s + C,
\end{equation}
where $A, C$ are constants and $W$ is the Lambert function. The coordinate $y(s)$ is determined, up to a constant, by 
$$y(s) = \int \pm e^{z(s)} \sqrt{1 - z'(s)^2}\, ds.$$
\end{theorem}

\begin{proof}
Suppose that $\gamma'$ is vertical at some point $s$. Then $\cos\theta(s)=0$ and \eqref{121} gives a contradiction. This proves that $\gamma$ is a graph on the $y$-axis. On the other hand, $z''=\theta'\cos\theta$ and $\cos^2\theta=1-z'^2$. Thus equation \eqref{121} becomes
\begin{equation}\label{211}
z'' - z'^2 - z' e^{-z} + 1 = 0.
\end{equation}
 Since the independent parameter $s$ does not appear explicitly in \eqref{211}, we can reduce its order by treating $z$ as the independent variable by introducing the function $v(z) = z'(s)$, with $v=v(z)$. Using the chain rule, the second derivative of $z$ with respect to $s$ can be rewritten as
 \begin{equation}\label{eq211}
v v' - v^2 - v e^{-z} + 1 = 0, \quad \text{with } v \neq 0.
\end{equation} 
By applying the change of variable $h(z) = v e^{-z}$, the ODE is 
$$hh'e^{2z}=h-1.$$
Solving by separation of variables, we arrive at 
$$
h + \log(h-1) = -\frac{1}{2}e^{-2z} + C_1.
$$
We now use the Lambert $W$ function, which is defined as the inverse of $x\mapsto xe^x$. Thus we need to write the above equation as
$$h-1 + \log(h-1) = -\frac{1}{2}e^{-2z} + C_1-1,$$
or equivalently, 
$$(h-1)e^{h-1}=A\exp(-\frac{1}{2}e^{-2z}),\quad A=e^{C_1-1}.$$
 Thus 
 $$h(z)=1+ W\left( A e^{-\frac{1}{2} e^{-2z}} \right).$$
Since $z' = v = he^z$, then 
\begin{equation*}
z' = e^z \left[ 1 + W\left( A \exp\left( -\frac{1}{2} e^{-2z} \right) \right) \right]
\end{equation*}
This gives \eqref{eqf1} by separation of variables.

\end{proof}

We now describe the geometric properties of the generating curves of $F_1$-invariant intrinsic $F_2$-solitons. We restrict this study to the case where the generating curve has a point with a horizontal tangent vector.

 \begin{theorem} \label{t55}
Let $\Sigma$ be an $F_1$-invariant extrinsic $F_2$-soliton generated by a curve $\gamma$ which is a solution of \eqref{1pp}--\eqref{121} with initial conditions $y(0)=y_0$, $z(0)=z_0$ and $\theta(0)=0$. Then $\gamma$ has the following geometric properties: 
\begin{enumerate}
 \item The curve $\gamma$ cannot be a horizontal line.
 \item The point $s=0$ is a strict global maximum for the height function $z(s)$. 
 \item For $s>0$, $\gamma$ ends at a finite arc-length and a finite height with vertical velocity ($\theta \to -\pi/2$).
 \item For $s<0$, $\gamma$ descends infinitely ($z \to -\infty$) and in the $(y,z)$ plane, $\gamma$ is asymptotic to a straight line of slope $ 1$.
\end{enumerate}
\end{theorem}

\begin{proof} 
As established in \eqref{eq211}, if $v=z'$, then $v=v(s)$ satisfies the second-order equation $v' = v^2 + v e^{-z} - 1$. Let us introduce the function $w(s) = e^{-z(s)} > 0$ and define a planar autonomous dynamical system by 
\begin{equation}\label{fase1}
\begin{cases}
v' = v^2 + vw - 1, \\
w' = -vw.
\end{cases}
\end{equation}
 
 The asymptotic behavior as $z \to +\infty$ corresponds to the limit $w \to 0$. The equilibrium points of the system are $P_1 = (1, 0)$ and $P_2 = (-1, 0)$. Since $v^2=1$, then $\theta(s)=\pm\frac{\pi}{2}$ and these points represent generating curves that are vertical lines. By computing the Jacobian matrix associated with \eqref{fase1}, it is not difficult to see that $P_1$ and $P_2$ are saddle points.

Denote by $\eta(s)=(v(s),w(s))$ the trajectory associated with a solution $(y(s),z(s),\theta(s))$ of \eqref{1pp}--\eqref{121}. In such a case, $(v,w)\in \Theta:=[-1,1]\times (0,\infty)$. Note that $P_1$ and $P_2$ belong to the boundary of $\Theta$.
\begin{enumerate}
\item If $\gamma$ is a horizontal line, then $v = 0$ and thus $v' = 0$. However, the first equation of \eqref{fase1} yields a contradiction. The result is also a consequence of Theorem \ref{t52}.
 \item At $s=0$, we have $z'(0)=0$ and by \eqref{211}, $z''(0)=-1<0$. Thus $s=0$ is a strict local maximum. Since $z''<0$ for any critical point, no more local maximum can exist by the Rolle's theorem, making $s=0$ the strict global maximum.

\item For $s > 0$, $v(s)$ initially becomes negative because $v'(0)=-1$. As long as $v \in (-1, 0)$ and $w > 0$, the first equation of \eqref{fase1} gives $v' = (v^2 - 1) + vw<0$. Furthermore, since $w' = -vw > 0$, $w$ is strictly increasing, making $v'$ even more negative. Consequently, $v(s)$ strictly decreases without asymptotic limits in $(-1, 0)$. Since $v(s) \ge -1$, the trajectory $\eta$ must reach the boundary $v = -1$ of $\Theta$. Evaluating the vector field at this boundary gives $v'_{|v=-1} = -w < 0$. Since $\eta$ is strictly transverse to $v=-1$, it reaches this state in finite time $s_1 > 0$ at some finite value $w_1 > w_0$. Geometrically, this implies $\gamma$ cannot be smoothly extended further attaining a height $z_1 = -\log(w_1)$. Moreover, since $v'\to -1$, then $\gamma$ becomes vertical.

\item For $s < 0$, since $v(0)=0$ and $v'(0)=-1$, the function $v(s)$ must be strictly positive. Since $0<v\leq 1$, then the second equation of \eqref{fase1} yields $\frac{w'}{w}\geq -1$. Thus $w(s)\geq w(0)e^{-s}$. Letting $s\to-\infty$, we derive $\lim_{s\to-\infty}w(s)=+\infty$, which implies $z \to -\infty$. 

{\it Claim:} $\lim_{s \to -\infty} vw = 1$. Let us define $f(s) = v(s)w(s)$. Differentiating $f$ with respect to $s$ and using the system \eqref{fase1}, we obtain
$$f'=(z''-z'^2)e^{-z}=e^{-z}(f-1).$$
 Integrating from $s$ to $0$, with $s < 0$, we get:
\begin{equation*}
\begin{split}
 \log|f(0) - 1| - \log|f(s) - 1| &= \int_{s}^{0} w(\tau) \, d\tau \geq w(0) \int_{s}^{0} e^{-\tau} \, d\tau\\
 &=w(0)(e^{-s}-1).
 \end{split}
 \end{equation*}
Since $f(0) = 0$, and $h'<0$ for $f<1$, it follows that $f(s) < 1$ for all $s < 0$. Thus
$$ -\log(1 - f(s)) \geq w(0)(e^{-s}-1).$$
Letting $s\to-\infty$, we conclude that $\lim_{s \to -\infty} f(s) = 1$. This proves the claim. 

Finally, from the claim, $\lim_{s\to -\infty}v(s) = 0$ because $\lim_{s\to-\infty}w(s)=+\infty$. The slope of $\gamma$, viewed as the graph $z=z(y)$, is
$$ \frac{dz}{dy} = \frac{z'(s)}{y'(s)} = \frac{v(s)}{ \frac{1}{w(s)} \sqrt{1-v(s)^2}} = \frac{v(s)w(s)}{\sqrt{1-v(s)^2}} \to 1,\quad\mbox{as $s\to-\infty$}.$$
Consequently, the branch of $\gamma$ for $s<0$ descends infinitely while becoming asymptotic to a straight line of slope $1$.
 
\end{enumerate}
\end{proof}

 \begin{remark}
The previous theorem describes the family of generating curves that possess a point with a horizontal tangent ($v=0$). We may use an argument by contradiction to prove that any trajectory must contain points with a horizontal tangent. The reasoning would be the following. Suppose there exists a trajectory $\eta(s)=(v(s),w(s))$ such that $v\not=0$ for all $s$. Without loss of generality, suppose $v>0$. By the second equation of \eqref{fase1}, $w$ is monotonic decreasing, and because $w>0$, there is some $w_\infty\geq 0$ such that $w\to w_\infty$ and $w'\to 0$. If $w_\infty>0$, 
the first equation of \eqref{fase1} implies that $v\to 0$. Now, the first equation of \eqref{fase1} 
yields $v'\sim -1$ at the limit. Thus, necessarily, $\eta$ must cross the line $v=0$, obtaining a contradiction. Thus, we 
have proved that $w_\infty=0$. However, we can neither deduce that $v\to 0$ nor establish 
the monotonicity of $v$ using the first equation of \eqref{fase1}. This is a situation that 
really can occur. This trajectory corresponds to the stable manifolds of the saddle points 
$P_1=(1,0)$ or $P_2=(-1,0)$. For the generating curve, a branch of it strictly ascends 
without ever attaining a global maximum. Moreover, its tangent vector tends to be vertical 
because $v \to 1$ ($\theta \to \pi/2$). An analogous strictly decreasing solution exists associated with $P_2$.
\end{remark}

 Now we do a similar treatment for intrinsic $F_2$-solitons.
 
\begin{theorem}\label{t57}
Let $\Sigma$ be an $F_1$-invariant surface in $\sol$. Then $\Sigma$ is an intrinsic $F_2$-soliton if and only if the height function $z(s)$ satisfies
\begin{equation} \label{eq27}
z'' + z'^2 - z' e^{-z} = 0.
\end{equation}
Moreover, either $\Sigma$ is a horizontal plane or $z(s)$ is implicitly determined by
\begin{equation}\label{eq24}
s = e^{-C} \operatorname{Ei}(z + C) + A,
\end{equation}
where $\operatorname{Ei}(u) = \int_{-\infty}^u \frac{e^t}{t} dt$ is the exponential integral function and $A$ is a constant. In the latter case, a branch of $\gamma$ is asymptotic to a horizontal line at finite height $z = -C$ and the height of the other branch goes to $+\infty$.
\end{theorem}

\begin{proof}

Using that $z'' = \theta'\cos\theta$, equation \eqref{eq27} is an immediate consequence of \eqref{122}. It is known by Table \ref{ex1} that the horizontal planes are intrinsic $F_2$-solitons.

We prove that if $\gamma$ has a point, say $s=0$ where $z'(0)=0$, then $\gamma$ is a horizontal line, hence $\Sigma$ is a horizontal plane. Without loss of generality, assume $\theta(0)=0$. In such a case, the triple of functions $(e^{z_0}+y_0,z_0,0)$ is a solution of the system with the case initial conditions, and by uniqueness, $\gamma(s) =(0,e^{z_0}+y_0,z_0)$. This proves that $\gamma$ is a horizontal line.

From the above argument, if $\Sigma$ is an $F_1$-invariant surface which is also an intrinsic $F_1$-soliton and if, in addition, $\Sigma$ is not a horizontal plane, then the coordinate $z(s)$ of its generating curve $\gamma$ satisfies $z'(s)\not=0$ for all $s$. Let $v(z) = z'(s)$. By the chain rule, $z'' = v \frac{dv}{dz}$, so equation \eqref{eq27} transforms into 
$$ v \frac{dv}{dz} + v^2 - v e^{-z} = 0. $$
Since $v\not=0$, dividing by $v$ yields the first-order linear ODE 
$$ \frac{dv}{dz} + v = e^{-z}. $$
Multiplying by the integrating factor $e^z$, the equation becomes $(v e^z)' = 1$. Integrating with respect to $z$, we find 
\begin{equation}\label{v2}
v(z) e^z = z + C,
\end{equation}
where $C \in \r$ is a constant. Thus, $z' = v(z) = (z+C)e^{-z}$. Separating variables and integrating yields 
\begin{equation}\label{z2}
\int ds = \int \frac{e^z}{z+C} dz.
\end{equation}
 The change of variable $u = z+C$ leads directly to 
$$ s = e^{-C} \int \frac{e^u}{u} du = e^{-C} \operatorname{Ei}(z + C) + A, $$
where $A \in \r$ is a constant of integration. This proves \eqref{eq24}.

We now examine the geometric properties of $\gamma$. For this, we use the properties of the function $\operatorname{Ei}(x)$ \cite{ab}. Since the integrand in the right-hand side of \eqref{z2} is not integrable as $z \to +\infty$, the arc length $s(z)$ diverges to $+\infty$. For large positive values of the argument, the asymptotic expansion of the exponential integral is $\operatorname{Ei}(x) \sim \frac{e^x}{x}$. Therefore, we have $z\to+\infty$ as $s(z)\to+\infty$. Furthermore, the derivative of $z$ is $z'= (z+C)e^{-z}$, which tends to $0$ as $z \to \infty$. This means that its slope becomes nearly horizontal at infinite height.

On the other hand, the condition $|z'|=|\sin\theta| \leq 1$ must hold. Equation \eqref{v2} implies that $z$ must satisfy $(z+C)e^{-z} \leq 1$ for all $z$ in the domain of the solution. Since the function $z\mapsto (z+C)e^{-z}$ attains its maximum value $e^{C-1}$ at $z = 1-C$, a global solution exists for all $z > -C$ if and only if $e^{C-1}\leq 1$, which leads to $C \leq 1$. Under this condition, the logarithmic divergence of the integral $\int \frac{e^z}{z+C} dz$ as $z \to -C^+$ proves that the generating curve is asymptotic to the horizontal line $z = -C$ as $s \to -\infty$. See Fig. \ref{fig1}, right. 
\end{proof}

 \begin{figure}
 \includegraphics[width=.35\textwidth]{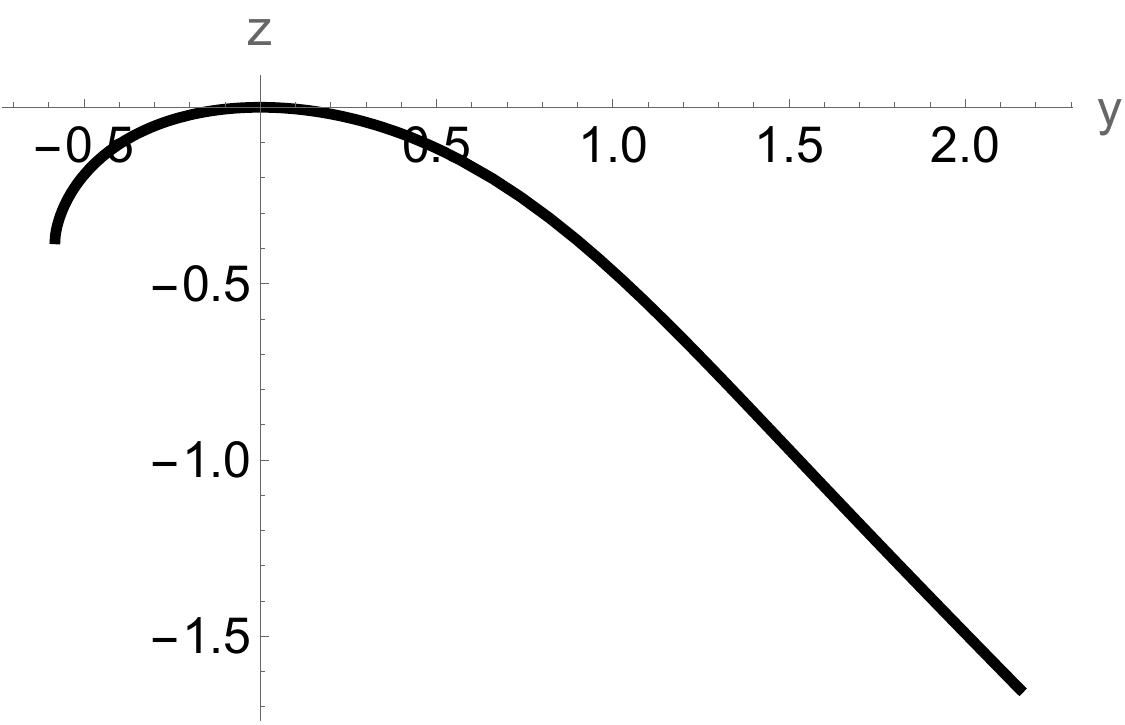}\quad \includegraphics[width=.55\textwidth]{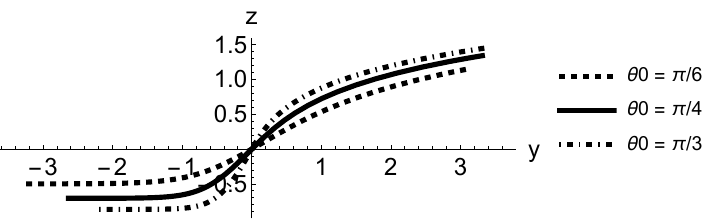} 
 \caption{$F_1$-invariant surfaces: the case of $F_2$-solitons: extrinsic curvature (left) and intrinsic curvature (right)}
 \label{fig1}
 \end{figure}

 \section{$F_1$-invariant surfaces: the case of $F_3$-solitons}\label{s6}

We first obtain the equations of extrinsic and intrinsic $F_3$-solitons in terms of the angle $\theta$ of the generating curve. 

\begin{proposition} Let $\Sigma$ be an $F_1$-invariant surface in $\sol$ generated by a curve parametrized by arc-length according to \eqref{1pp}. Suppose that $\Sigma$ is an $F_3$-soliton.
\begin{enumerate}
\item If $\Sigma$ is an extrinsic soliton, then 
 \begin{equation}\label{131}
 \theta' \cos\theta= y\sin\theta\, e^{-z} - \cos\theta - \cos^2\theta.
 \end{equation}
\item If $\Sigma$ is an intrinsic soliton, then 
 \begin{equation}\label{132}
 \theta'\cos\theta = y \sin\theta\, e^{-z} - \cos\theta-\sin^2\theta.
 \end{equation}
\end{enumerate}
\end{proposition}

From now on, we will assume that the generating curves of $F_1$-invariant surfaces are parametrized according to \eqref{1pp}. 

In the following, we study the case where the function $\theta'$ is constant. First, we assume $\theta(s)=c$ (Theorem \ref{t62}) and then $\theta(s)=cs$ (Theorem \ref{t63}).

\begin{theorem}\label{t62}
Let $\Sigma$ be an $F_1$-invariant surface in $\sol$ generated by a curve $\gamma(s)$. Suppose that $\theta(s)=c$.
\begin{enumerate}
 \item If $\Sigma$ is an extrinsic $F_3$-soliton, then $\Sigma$ is either a horizontal plane $z = z_0$ or the vertical plane $y = 0$.
 \item The surface $\Sigma$ cannot be an intrinsic $F_3$-soliton.
\end{enumerate}
\end{theorem}

\begin{proof}
\begin{enumerate}
\item Suppose $\Sigma$ is an extrinsic $F_3$-soliton. If $\theta = c$, then \eqref{131} gives
$$\cos^2 c = y(s)\sin c \, e^{-z(s)} - \cos c.$$
We analyze two cases based on $\sin c$.
\begin{enumerate}
 \item If $\sin c = 0$, then $c = 0$ or $c = \pi$. The equation reduces to $\cos^2 c =- \cos c$. Since $\cos c \in \{1, -1\}$, the only valid solution is $\cos c = -1$, which implies $c = \pi$. Consequently, $z'(s) = \sin \pi = 0$, so $z(s) = z_0$ for some constant $z_0$. The condition $y'(s) = e^{z_0}\cos 0$ integrates to $y(s) =- s e^{z_0} + y_0$. This generates a horizontal plane $z = z_0$.
 \item If $\sin c \neq 0$, we can isolate $y(s)$:
 $$y(s)e^{-z(s)} = \frac{\cos c - \cos^2 c}{\sin c} \equiv C,$$
 where $C$ is a constant. Thus, $y(s) = C\, e^{z(s)}$. Differentiating with respect to $s$ gives $y'(s) = C z'(s) e^{z(s)} = C \sin c \, e^{z(s)}$. However, from the arc-length parametrization, we know $y'(s) = \cos c \, e^{z(s)}$. Equating both expressions yields $\cos c = C \sin c$. Substituting the value of $C$ into this relation gives, after simplification, $ \cos^2 c = 0$. 
 Therefore, $c = \pm \pi/2$ and $\sin c = \pm 1$. Since $\cos c = 0$, we have $C = 0$, which implies $y(s) = 0$. This generates the vertical plane $y = 0$.
\end{enumerate}
\item Suppose $\Sigma$ is an intrinsic $F_3$-soliton. If $\theta = c$, then \eqref{132} becomes
$$\sin^2 c = y(s)\sin c \, e^{-z(s)} - \cos c.$$
Again, we analyze two cases:
\begin{enumerate}
 \item If $\sin c = 0$, then $\cos c = \pm 1$. The equation becomes $0 = \pm 1$, which is a contradiction.
 \item If $\sin c \neq 0$, then
 $$y(s)e^{-z(s)} = \frac{\cos c (1+\sin^2 c)}{\sin c} \equiv C,$$
 for some constant $C$. Thus, $y(s) = C\, e^{z(s)}$. Differentiating yields $y'(s) = C \sin c \, e^{z(s)}$. Equating this to $y'(s) = \cos c \, e^{z(s)}$ gives $\cos c = C \sin c$. Substituting $C$ back into this relation yields $ \sin^2 c = 0$, 
 which contradicts our assumption that $\sin c \neq 0$. Thus, no such surface exists.
\end{enumerate}
\end{enumerate}
\end{proof}

\begin{theorem}\label{t63}
 Let $\Sigma$ be an $F_1$-invariant surface in $\sol$ generated by a curve $\gamma$. Suppose the angle function satisfies $\theta'(s) = c$ for some non-zero constant $c \in \r$. Then $\Sigma$ can be neither an extrinsic nor an intrinsic $F_3$-soliton.
\end{theorem}

\begin{proof}
Assume $\theta'(s) = c \neq 0$, which implies that $\theta(s) = cs + \theta_0$ is non-constant. 
 
\begin{enumerate}
\item Suppose $\Sigma$ is an extrinsic $F_3$-soliton. 
Substituting $\theta' = c$ in \eqref{131} yields 
$$y e^{-z} \sin\theta = (c+1)\cos\theta + \cos^2\theta.$$
Differentiating this equation with respect to $s$ and simplifying gives
$$\cos\theta \sin^2\theta + \left( (c+1)\cos\theta + \cos^2\theta \right)(c\cos\theta - \sin^2\theta) = -c(c+1)\sin^2\theta - 2c\cos\theta\sin^2\theta,$$
or equivalently, $ 8 c^2+8 c \cos \theta +8 c+\cos (4 \theta )-1 = 0$. This yields a contradiction.

\item Suppose $\Sigma$ is an intrinsic $F_3$-soliton. 
Substituting $\theta' = c$ in \eqref{132} yields
$$y e^{-z} \sin\theta = (c+1)\cos\theta + \sin^2\theta.$$
Differentiating this with respect to $s$ gives 
$$-\cos^4\theta + 2c\cos^3\theta + 2\cos^2\theta - 2c\cos\theta + (c^2+c-1) = 0.$$
This can be written as 
$$ 8 c^2-4 c \cos \theta +4 c \cos (3 \theta )+8 c+4 \cos (2 \theta )-\cos (4 \theta )-3 =0,$$
obtaining a contradiction.
\end{enumerate}
\end{proof}

In Table \ref{table2}, we show all cases of $F_1$-invariant solitons where the angle function $\theta$ is linear. 

\begin{table}[h!]
\centering
\caption{Existence of $F_1$-invariant solitons with a linear function angle $\theta(s)$.}\label{table2}
\begin{tabular}{lll}
\hline
Soliton type & Case $\theta' = 0$ & Case $\theta' = c \neq 0$ \\
\hline
Extrinsic $F_2$ & Does not exist & Does not exist \\
Intrinsic $F_2$ & Only horizontal planes & Does not exist \\
Extrinsic $F_3$ & Only horizontal and vertical planes & Does not exist \\
Intrinsic $F_3$ & Does not exist & Does not exist \\
\hline
\end{tabular}
\end{table}

We now consider the general case for the function $\theta(s)$. A key difference from the $F_2$-soliton case 
is that the function $y$ appears explicitly in \eqref{131} and \eqref{132}, in contrast to 
Eqs. \eqref{121} and \eqref{122}. In the following result, we consider extrinsic solitons.

\begin{theorem}\label{t64}
Let $\gamma$ be the generating curve of an $F_1$-invariant extrinsic $F_3$-soliton. 
\begin{enumerate}
 \item Horizontal planes $z = z_0$ are extrinsic $F_3$-solitons, which correspond to $\theta(s)=\pi$ in \eqref{131}.
 \item The plane $y=0$ is the unique solution generated by a vertical line. Any other generating curve is a graph over the $y$-axis. 
 
 \end{enumerate}
\end{theorem}
 
\begin{proof}
\begin{enumerate}
\item This follows directly from Theorem \ref{t62}. Note that, viewing a horizontal line as a solution of \eqref{1pp}-\eqref{131}, then $\theta(s)=\pi$.
\item We know that the plane $y=0$ is an extrinsic $F_3$-soliton by Theorem \ref{t62}. Suppose that $\gamma$ is a generating curve such that at $s=0$, $\gamma'(0)$ is a vertical vector. Then the triple of functions $(y(s),z(s),\theta(s))$ is a solution of \eqref{1pp}-\eqref{131} with initial conditions $y(0)=y_0$, $z(0)=z_0$ and $\theta(0)=\pi/2$ or $\theta(0)=-\pi/2$. Then \eqref{131} at $s=0$ implies $y_0e^{-z_0}=0$, hence $y_0=0$. Now, it is immediate that the triple of functions $(0,\pm s+z_0,\pm\frac{\pi}{2})$ is a solution of \eqref{1pp}-\eqref{131} with the same initial conditions. By uniqueness, $\gamma(s)=(0,0,\pm s+z_0)$, proving that $\gamma$ is the vertical line $y=0$. 

As a consequence, any other solution must be a graph over the $y$-axis, proving the result. 
\end{enumerate}
\end{proof}

The variety of generating curves depends on the initial conditions for the system \eqref{1pp}-\eqref{131}. This can be, in part, simplified by defining the following dynamical system and studying its phase portrait. Let $u(s) = y(s)e^{-z(s)}$. 
By differentiating $u$, we find $u' = y' e^{-z} - y z' e^{-z} = \cos\theta - u\sin\theta$. Similarly, \eqref{131} can be written in terms of $u$ as $-\theta'\cos\theta = -u \sin\theta + \cos\theta + \cos^2\theta$. We have constructed the dynamical system
\begin{equation}\label{fase2}
\begin{cases}
u' = \cos\theta - u\sin\theta \\
\theta' = u\tan\theta - 1 - \cos\theta, \quad \theta \neq \pm \pi/2.
\end{cases}
\end{equation} 
To analyze the behavior at the singular lines $\theta = \pm \pi/2$, we introduce the time reparametrization $d\tau = \frac{ds}{\cos\theta}$ to desingularize the vector field, obtaining the smooth system:
\begin{equation}\label{fase3}
\begin{cases}
\frac{du}{d\tau} = \cos^2\theta - u\sin\theta\cos\theta \\
\frac{d\theta}{d\tau} = u\sin\theta - \cos\theta - \cos^2\theta.
\end{cases}
\end{equation}
 Setting the derivatives to zero yields the equilibria $P_1=(0, \pi/2)$ and $P_2=(0, -\pi/2)$, as well as their $2\pi$-multiplies in the second coordinate. These points correspond strictly to the vertical plane $y=0$ (Theorem \ref{t62}). In Fig. \ref{fig2}, we depict the phase plane, indicating the vertical line $u=0$, and the horizontal lines $\theta=\pm\pi/2$ and $\theta=\pm\pi$.

The Jacobian matrix of \eqref{fase3} is
$$J(u, \theta) = \begin{pmatrix} -\sin\theta\cos\theta & -2\sin\theta\cos\theta - u(\cos^2\theta-\sin^2\theta) \\ \sin\theta & u\cos\theta + \sin\theta + 2\cos\theta\sin\theta \end{pmatrix}.$$
Evaluating at $P_1$ and $P_2$, we obtain:
$$J(P_1) = \begin{pmatrix} 0 & 0 \\ 1 & 1 \end{pmatrix},\quad J(P_2) = \begin{pmatrix} 0 & 0 \\ -1 & -1 \end{pmatrix}.$$
The eigenvalues are $\{1, 0\}$ for $P_1$ and $\{-1, 0\}$ for $P_2$. In particular, $P_1$ has an unstable manifold for the eigenvalue $1$ and $P_2$ a stable manifold for the eigenvalue $-1$. 

Recall that $\cos\theta\not=0$, and thus, although the trajectories cross the horizontal line $\theta=\pm\pi/2$, the generating curve cannot be extended further. For example, if $u=0$, which is equivalent to considering $y_0=0$, and if $\theta(0)\in(-\frac{\pi}{2},\frac{\pi}{2})$, the trajectory in the phase plane crosses the lines $\theta=\pm\pi/2$ for some values of $u$. This implies that the generating curve is defined in some bounded interval of the $y$-line (because $u$ is finite) and, at the endpoints, the curve $\gamma$ has vertical tangent vectors. See Fig. \ref{fig2}, left.

 \begin{figure}
 \includegraphics[width=.5\textwidth]{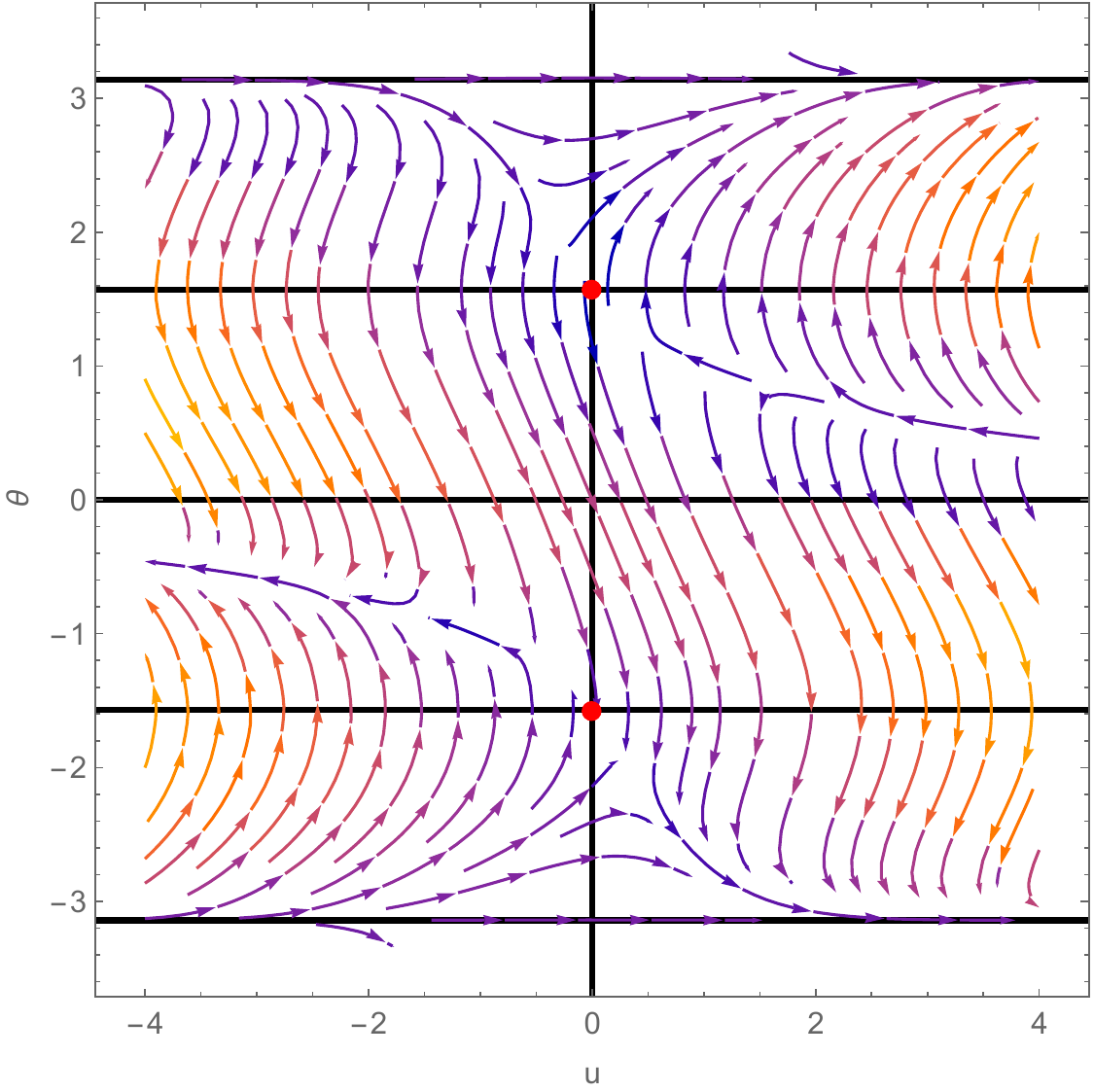} 
 \caption{The phase portrait of \eqref{fase3}.}
 \label{fig2}
 \end{figure}
 
However, there are other trajectories that cross $\theta=\pi/2$ but not $\theta=-\pi/2$, with the other branch of the trajectory being asymptotic to the line $\theta=\pi$. This implies that the solution ends vertically at a finite parameter $s$ for one branch, but the other branch goes to $y\to\infty$. Also, other trajectories remain in the strip $\pi/2<\theta<\pi$, which implies that $\gamma$ is a graph over the entire $y$-axis, being asymptotic to horizontal line at infinity. See Fig. \ref{fig3}. 

 \begin{figure}
 \includegraphics[width=.3\textwidth]{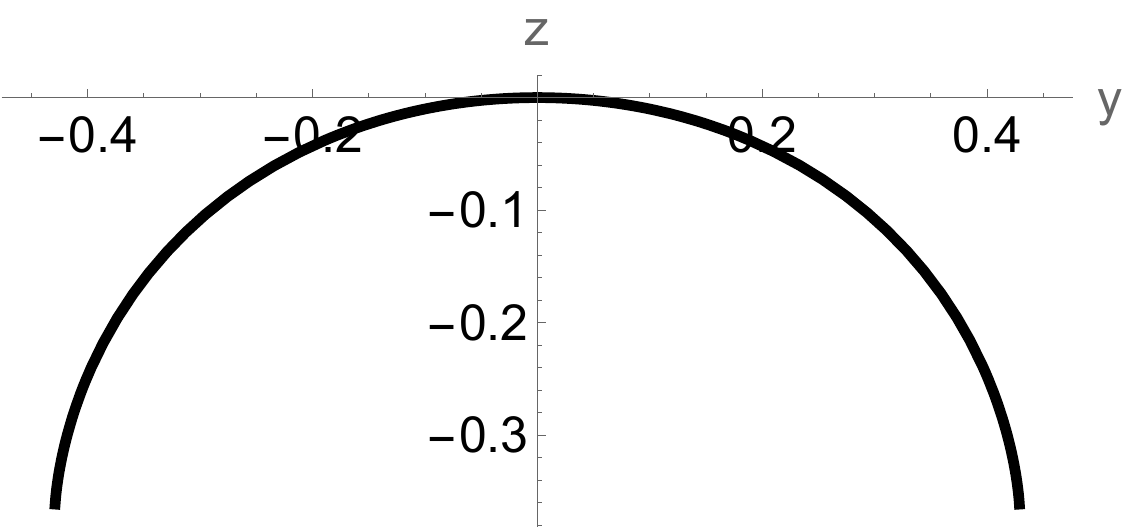} \quad \includegraphics[width=.25\textwidth]{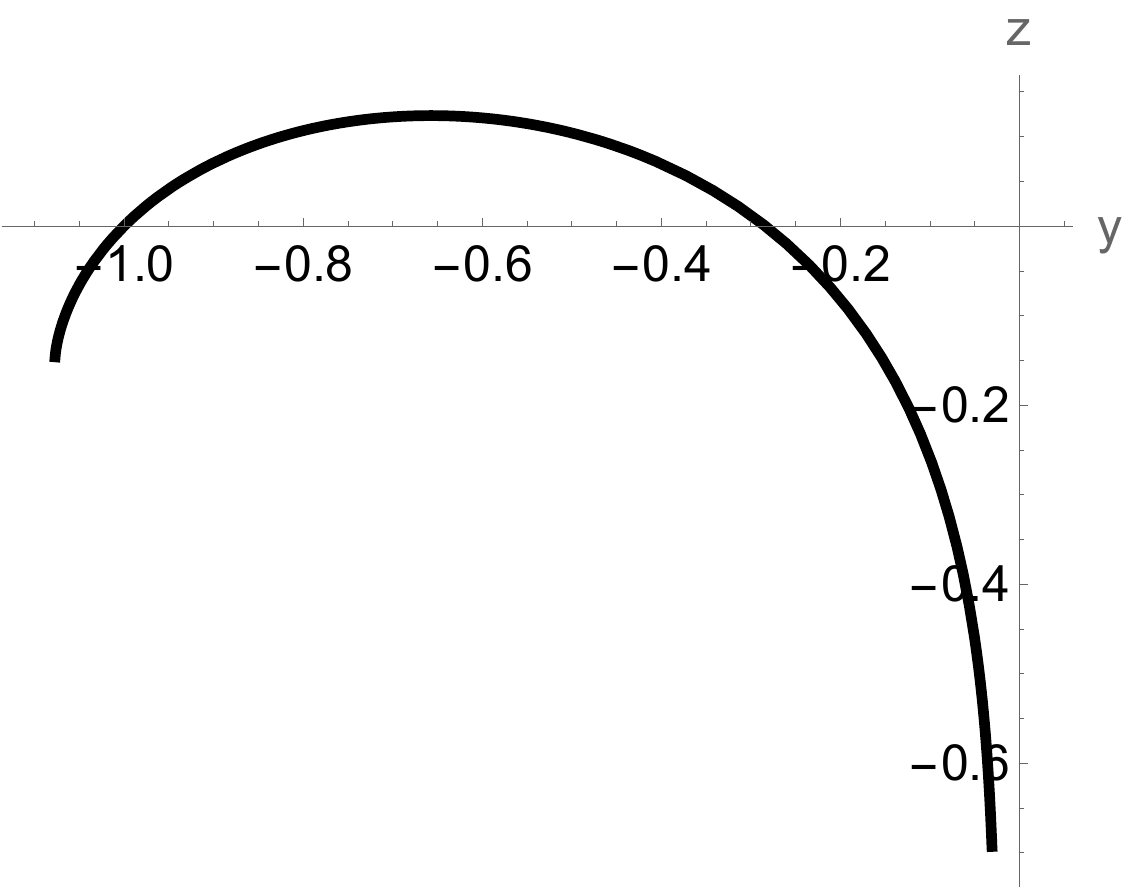}\quad \includegraphics[width=.3\textwidth]{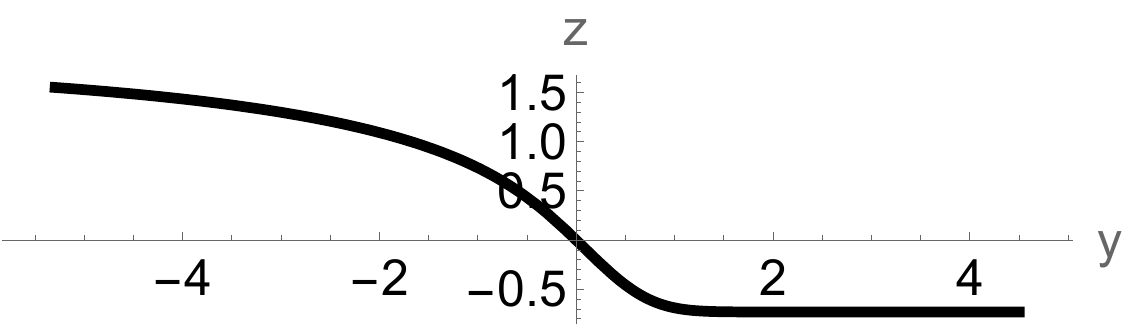}
 \caption{Generating curves of $F_1$-invariant extrinsic $F_3$-solitons.}
 \label{fig3}
 \end{figure}

 We now consider intrinsic $F_3$-solitons. 
 
 \begin{theorem}\label{t65}
Let $\Sigma$ be an $F_1$-invariant surface in $\sol$. Then $\Sigma$ is an intrinsic $F_3$-soliton if and only if the angle $\theta(s)$ and the scaled horizontal coordinate $u(s) = y(s)e^{-z(s)}$, satisfy the following autonomous planar system:
\begin{equation} \label{fase4}
\begin{cases}
u' = \cos\theta - u\sin\theta \\
\theta' = u\tan\theta - 1 - \sin\theta\tan\theta, \quad \text{with } \theta \neq \pm \pi/2.
\end{cases}
\end{equation}
Furthermore: 
\begin{enumerate}
 \item The surface $\Sigma$ cannot be a horizontal plane. 
 
 \item There are generating curves that reach a vertical tangency in finite arc length $s$. At these points, the solution cannot be extended.
\item There exist solutions that are entire graphs over the $y$-axis.
\end{enumerate}

\end{theorem}

\begin{proof}
The derivation of \eqref{fase4} is analogous to that of the previous theorem. We consider \eqref{fase4}. Since it is singular at $\theta=\pm\pi/2$, we apply the reparametrization $d\tau = \frac{ds}{\cos\theta}$ to obtain a smooth, globally defined vector field
\begin{equation}\label{fase5}
\begin{cases}
\frac{du}{d\tau} = \cos^2\theta - u\sin\theta\cos\theta \\
\frac{d\theta}{d\tau} = u\sin\theta - \cos\theta - \sin^2\theta
\end{cases}
\end{equation}
To find the equilibria, we set the derivatives to zero, finding that the equilibrium points are $Q_1=(1, \pi/2)$ and $Q_2=(-1, -\pi/2)$. See Fig. \ref{fig4}. 

 \begin{figure}
 \includegraphics[width=.5\textwidth]{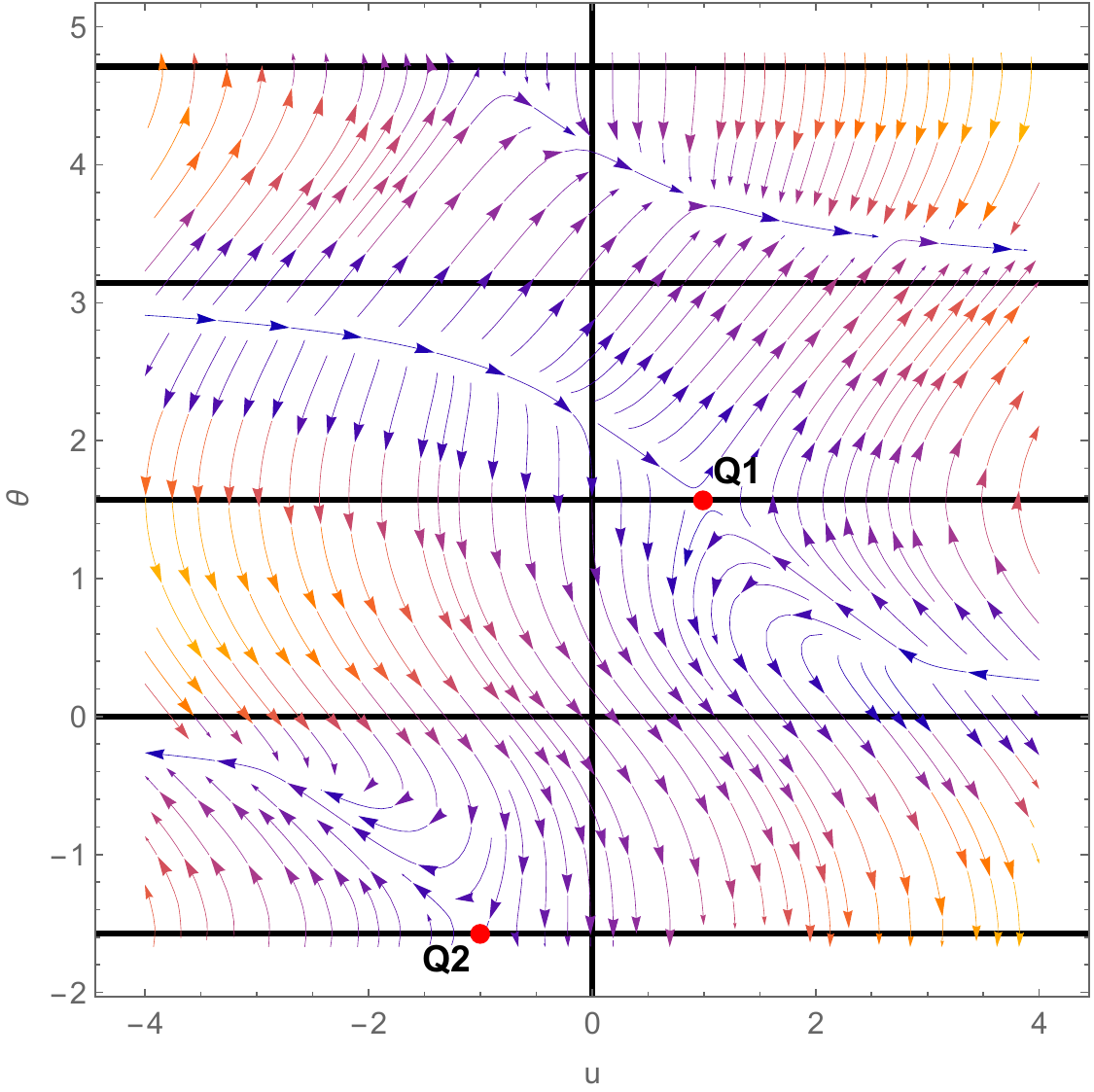} 
 \caption{The phase portrait of \eqref{fase5}, indicating the horizontal lines $\theta=\pm\frac{\pi}{2}, 0, \pi$ and $\frac{3\pi}{2}$. Note that there are trajectories that remain in the strip $\frac{\pi}{2}<\theta<\frac{3\pi}{2}$. However, no such trajectories remain in the strip $-\frac{\pi}{2}<\theta<\frac{\pi}{2}$. }
 \label{fig4}
 \end{figure}
 
 To determine the stability, we compute the Jacobian matrix of \eqref{fase5}:
$$J(u, \theta) = \begin{pmatrix} -\sin\theta\cos\theta & -2\sin\theta\cos\theta - u(\cos^2\theta-\sin^2\theta) \\ \sin\theta & u\cos\theta + \sin\theta - 2\sin\theta\cos\theta \end{pmatrix}$$
Evaluating at $Q_1$ and $Q_2$, we obtain
$$J(Q_1) = \begin{pmatrix} 0 & 1 \\ 1 & 1 \end{pmatrix},\qquad J(Q_2) = \begin{pmatrix} 0 & -1 \\ -1 & -1 \end{pmatrix}.$$
Both points are saddle points because their eigenvalues are real numbers with opposite signs.

Unlike the extrinsic case, the equilibrium points $Q_1$ and $Q_2$ do not correspond to any exact geometric solution. Indeed, a constant solution at $Q_1$ would require $\theta = \pi/2$ (hence $y'=0$ and $z'=1$) and simultaneously $u=1$ (hence $y = e^z$). Differentiating the latter yields $y' = e^z \neq 0$, which is a contradiction.

\begin{enumerate}
\item This is a consequence of Theorem \ref{t62}.

\item Consider initial conditions $\theta(0)=0$, $y(0)=0$. Then the corresponding trajectory in the phase plane crosses the lines $\theta=\pm\pi/2$. However, the solution $\gamma$ cannot be extended past $\cos\theta=0$. Since this occurs at some finite values of $y$, then $\gamma$ is defined on a a finite interval and, at the endpoints, the tangent vectors are vertical (Fig. \ref{fig5}, left).

\item Trajectories that remain entirely within the strip $\pi/2 < \theta < 3\pi/2$ for all $\tau \in \mathbb{R}$ correspond to curves where $\cos\theta$ remains strictly negative. For these trajectories, as $\tau \to \pm \infty$, the scaled coordinate $|u|$ diverges and the angle $\theta$ approaches $\pi$. Since $\cos\theta \to -1$ in both limits, the relation $ds = \cos\theta \, d\tau$ implies that the integral for the arc length diverges in both directions. Thus, the generating curve is defined for all $s \in \mathbb{R}$ and, since $y'(s) = e^z \cos\theta$ never vanishes and does not change sign, the solution constitutes an entire graph over the $y$-axis. See Fig. \ref{fig5}, center.
\end{enumerate}
\end{proof}
 
 Unlike the extrinsic case, the generating curve may not be a graph over the $y$-axis. For example, if at some point the tangent vector is vertical, this implies $\theta=\pm\pi/2$. Then \eqref{132} implies $ye^{-z}=\pm 1$ at that point. See Fig. \ref{fig5}, right. In the phase plane of \eqref{fase5} (Fig. \ref{fig4}), the corresponding trajectory comes from $u=\infty$ towards the equilibrium point $Q_1$ and then comes back to $u\to\infty$ remaining the trajectory in a horizontal strip.

 \begin{figure}
 \includegraphics[width=.3\textwidth]{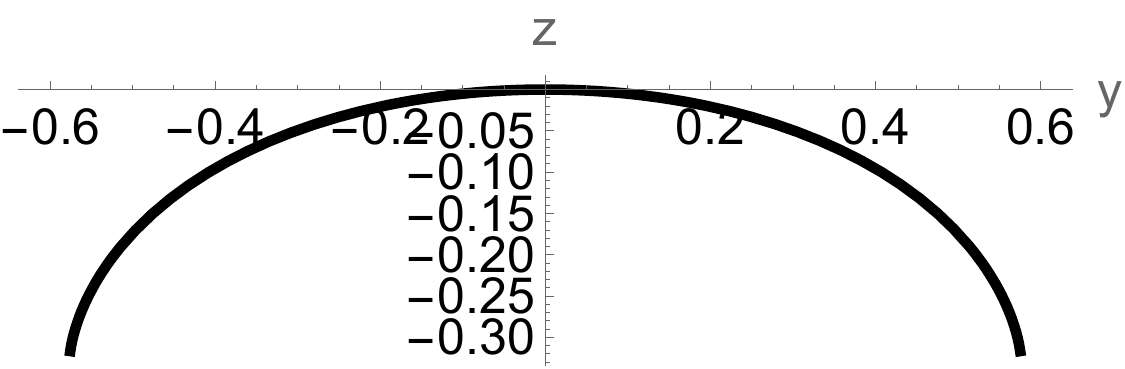} \quad \includegraphics[width=.3\textwidth]{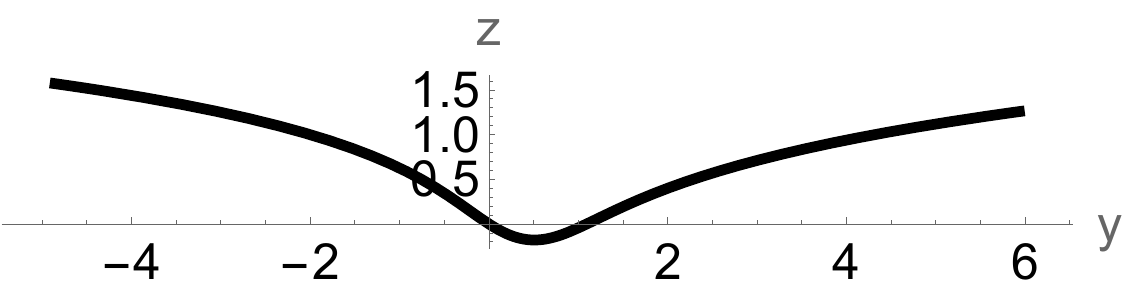}\quad \includegraphics[width=.3\textwidth]{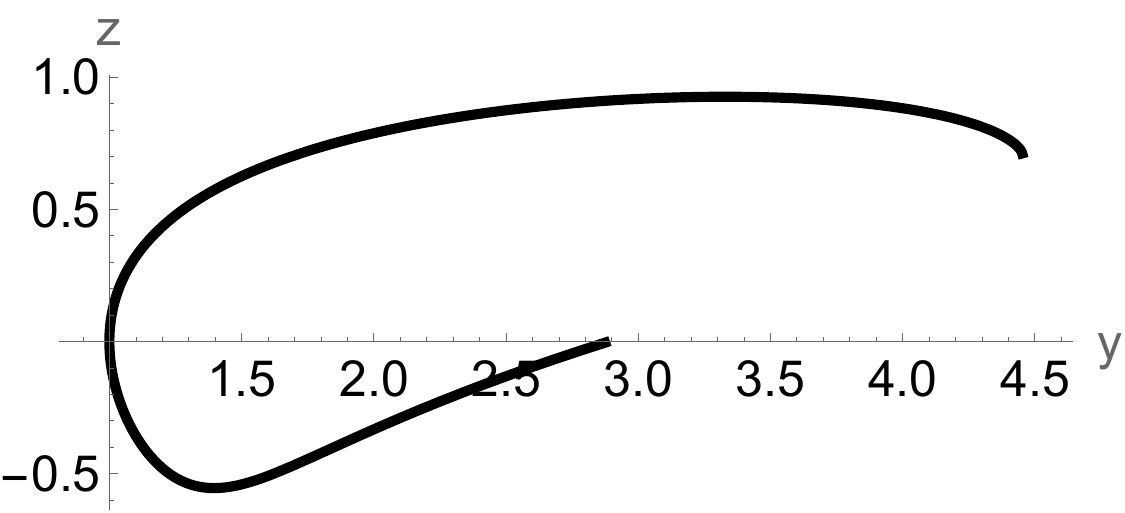}
 \caption{Generating curves of $F_1$-invariant intrinsic $F_3$-solitons.}
 \label{fig5}
 \end{figure}

 \section*{Acknowledgements}
Rafael Belli thanks the Department of Geometry and Topology of the University of Granada for their hospitality, where this work was carried out. 
Rafael L\'opez has been partially supported by MINECO/MICINN/FEDER grant
no. PID2023-150727NB-I00, and by the ``Mar\'{\i}a de Maeztu'' Excellence
Unit IMAG, reference CEX2020-001105- M, funded by MCINN/AEI/10.13039/
501100011033/ CEX2020-001105-M.

\section*{Ethics declarations}

Conflict of interest. The authors declare that they have no conflict of interest. No datasets were generated or analysed during the current study.


\end{document}